\documentclass{amsart}[12pt] 
\usepackage{amssymb, amscd, amsmath, amsthm, latexsym, hyperref}
\usepackage{pb-diagram, fancyhdr, graphicx, psfrag}
 \usepackage[text={6.3in,8.5in},centering,letterpaper,dvips]{geometry}
 \usepackage{relsize}
 
\newtheorem{thm}{Theorem}[section]
\newtheorem{lem}[thm]{Lemma}
\newtheorem{prop}[thm]{Proposition}
\newtheorem{cor}[thm]{Corollary} 
 
\newtheorem{rem}[thm]{Remark}

\def\trans{ \textsf{T}}
\def\rr{\mathbb{R}}
\def\cc{\mathbb{C}}
\def\zz{\mathbb{Z}}

\def\dd{\mathbb{D}}

\def\d{\frak{d}}
\def\i{\sqrt{-1}}
\DeclareMathOperator{\Si}{Sign}

\DeclareMathOperator{\mult}{mult}
\DeclareMathOperator{\ju}{jump}
\def\ooplus{{\mathlarger{\mathlarger{\mathlarger{\oplus}}}}}

\begin{document}

\title[Signature Jumps and Alexander Polynomials]{Signature Jumps and Alexander Polynomials for Links}

\author{Patrick M. Gilmer}
\address{Department of Mathematics\\
Louisiana State University\\
Baton Rouge, LA 70803\\
USA}
\email{gilmer@math.lsu.edu}
\urladdr{www.math.lsu.edu/\textasciitilde gilmer/}

\author{Charles Livingston}
\address{Department of Mathematics\\
Indiana University\\  
Bloomington, IN 47405\\ U.S.A.}
\email{livingst@indiana.edu}
\urladdr{mypage.iu.edu/~livingst/}

\thanks{The first author was partially supported by  NSF-DMS--1311911.  The second author was partially supported by a grant from the Simons Foundation and by NSF-DMS--1505586. }


\begin{abstract} 
We relate the jumps of the signature function of a link to the roots of its first non-zero higher Alexander polynomial that lie on the unit circle.
\end{abstract}

\maketitle 
 
 
 \section{Introduction} \label{sec.intro}
 
 A well-known result states that for a knot $K \subset S^3$, the 
 absolute value of the
 classical Murasugi signature, $\sigma_K$, is bounded above by the number of roots of the
 Alexander polynomial, $\Delta_K(t)$, on the unit circle, counted with multiplicity.  This is most easily proved using Matsumoto's result~\cite{matumoto} that the signature is   the 
 sum of Milnor signatures~\cite[\S 5]{milnor} of the knot at points on the upper half circle. The Milnor signature is the signature of a symmetric bilinear form on a space of dimension given by twice the multiplicity of the root.  A similar argument yields a generalization.\vskip.05in
 
 \noindent{\bf Theorem.} {\it For any $\omega \in S^1 \subset \cc$, if $\Delta_K(\omega) \ne 0$, then the 
 absolute value of the Levine-Tristram signature,    $\sigma_K (\omega)$, is bounded above by  the number of roots (counted with multiplicity) of the Alexander polynomial at points on the unit circle with real parts greater than that of $\omega$.}
 
 \vskip.05in
 
Generalizing this to links is nontrivial.  The Alexander polynomial of a link $L$ can be identically zero, and if nonzero, it can have roots at $-1$, which is problematic since $\sigma_L = \sigma_L(-1)$.  Also, following some of the approaches that work in the case of knots leads to subtle technical points that require lengthy arguments to overcome.   Our purpose here is to present an  approach to a generalization based on simultaneous row and column operations.
 
\vskip.05in 
  
\noindent {\it Acknowledgments}  We wish to thank Alexander Stoimenow for early discussions related to this material.   We are especially grateful to the referee for comments 
 that significantly improved the paper.
 
  
\section{Statements of results}\label{sec.results}

Let $L$ be an oriented link in $S^3$, let $\mu_L$ be the number of components of $L$, and let $\Delta_L$ denote the Alexander polynomial of
 $L$.  One also has higher Alexander polynomials, $\Delta_i(L)$, where $\Delta_L=\Delta_1(L)$. The polynomial $\Delta_i(L)$ can be defined  as the greatest common divisor of the set of all $(2g+\mu_L-i)$--minors of $tV-V^\trans$, where $V$ is a Seifert matrix for $L$ associated to a connected  Seifert surface of genus $g$; these are well-defined up to multiplication by $\pm t^{k}$ for some $k$.   Let $A_L \in \zz[t,t^{-1}]$ be the first of these higher Alexander polynomials of $L$ which is nonzero.  By convention, a $0$--minor is taken to be $1$.  Thus every link $L$ has $\Delta_{i}(L)= 1$ for some $i$.   If $\Delta_L \ne 0$, then $A_L =\Delta_L$.

If $V$ is a Seifert matrix for $L$ coming from a connected Seifert surface, we consider 
$$W(t) = (1-t) V + (1- t^{-1} )V^\trans.$$   
Define the Levine-Tristram signature function~\cite{levine, tristram} on $S^1 \subset \cc$ by  $\sigma_L(\omega) = \Si(W(\omega))$, where  $\omega \in S^1 \subset \cc$. 
Note that one always has $\sigma_L(1)=0.$ The  Murasugi signature~\cite{murasugi} is denoted $\sigma_L$ and is defined to be $\sigma_L(-1)$.
We also consider one-sided limits, $\sigma_L^\pm(e^{2 \pi i x})= \lim_{y \to x^\pm} \sigma_L(e^{2 \pi i y})$, which in turn can be used to define one-sided jumps: $$\ju^\pm(\omega)= \pm( \sigma_L^\pm(\omega)- \sigma_L(\omega)).$$
Thus, $\ju^-(\omega)$ can be thought of as the jump in signature when arriving at $\omega$ while traveling counterclockwise around the circle. Similarly,
$\ju^+(\omega)$ should be thought of as the jump in signature when departing from  $\omega$.
We let $\mult_\omega(f(t))$ denote the multiplicity of $\omega$ as a root of 
 a polynomial 
  $f(t).$
  
\begin{thm}\label{main}
The  signature function $\sigma_L(\omega)$ is a step function on $S^1$ which can have   discontinuities only at  roots of $(t-1) A_L(t)$.
If $\omega \ne \pm 1$, then $|\ju^\pm(\omega)| \le \mult_\omega (A_L)$.  In addition,  $\mult_{-1}(A_L)$ is even,   $|\ju^\pm(-1)| \le (1/2)\mult_{-1} (A_L)$, and  $|\ju^\pm(1)| \le \mu_L -1$.
\end{thm}

This easily implies:   
\begin{cor}\label{bound}
The sum $|\sigma_L| +1 -\mu_L +(1/2) \mult_{-1}(A_L) $ is less than or equal to the number of roots of $A_L$ away from $1$ on the unit circle, counted with multiplicity.
\end{cor}
 
We note that if $L$ is a knot, then both  $\Delta_L \ne 0$ and $\mult_{-1}(\Delta_L)=0$.  
We obtain the following generalization of the result mentioned in the first sentence of this paper. 

\begin{thm}\label{bound2}
Assume $\Delta_L(t) \ne 0$. The  signature function $\sigma_L(\omega)$  can have   discontinuities only at  roots of $\Delta_L(t)$. Also $|\sigma_L| + (1/2) \mult_{-1}(\Delta_L(t)) $ is less than or equal to the number of roots of $\Delta_L$  on the unit circle counted with multiplicity.    
\end{thm}

We also wish to consider the total jump at $\omega$: $\ju(\omega)= \ju^+(\omega)+\ju^-(\omega)$.

\begin{thm}\label{cong}
If $\omega \ne \pm 1,$ $\ju(\omega) = 2\mult_\omega (A_L)$ modulo 4.   If $\omega \ne \pm 1$ 
and $\mult_{\omega} (A_L)=1$, then $\ju^+(\omega)=  \ju^-(\omega)= \pm 1.$
\end{thm}

  Let $h_L$ be the least $i$ such that $\Delta_i(L) \ne 0$.

   \begin{prop}\label{even}
For a link $L$,  $\mult_1(A_L)+\mu_L +h_L$ is even.
  \end{prop}

We prove  Theorem~\ref{main} in~\S \ref{sec.main}  and prove  Theorem~\ref{bound2} in~\S \ref{sec.bound2}. 
We discuss the proof of Theorem~\ref{cong} in~\S \ref{sec.cong}.
We give a proof of Proposition  \ref{even} in~\S \ref{evenp}.
In~\S \ref{module}, we discuss further restrictions on the signature function imposed by the 
structure  of the $\rr[t,t^{-1}]$--module presented by
the matrix $tV-V^\trans$.
In the last section, we illustrate these results by studying the signature functions of some links. For several of these examples, the signature function jumps at $-1$.

Our investigation of the signature function of links  began when   Stoimenow enquired, in connection with his work on~\cite{stoimenow}, whether  $\Delta_L \ne 0$ implies that  $|\sigma_L|$ is less than or equal to the number of roots of $\Delta_L$  on the unit circle, counted with multiplicity.  From the published version of~\cite{stoimenow}, we learned of~\cite{liechti}, in which an appendix by Peter Feller and Livio Liechti affirmatively answers the question that  Stoimenow posed.    Feller and Liechti's  elegant  proof   ignored the possibility of a discontinuity of the signature function at $-1$, but this omission may be easily repaired.  We were also influenced by Stavros Garoufalidis's work~\cite[Lemma 2.1]{garoufalidis} which discusses some of these results for knots. 
 The method of proof used here is quite different from the methods used by  Garoufalidis, Stoimenow and   Feller-Liechti.  
A similar argument to that used in the  proof of  Lemma~\ref{diag} was sketched 
by Kate  Kearney  in~\cite{kearney}.

 
\section{Diagonalizing Hermitian matrices over $\rr[t,t^{-1}]$ localized at certain prime ideals} \label{dp}

Let $\Lambda= \rr[t,t^{-1}]$ with the involution denoted ``bar''  determined by  $\overline{t} = t^{-1}$ and $\bar x=x$ for $x \in \rr.$  For $\rho \in S^1 \setminus \{\pm 1\}$,   
 let $f_\rho=  t+t^{-1} -(\rho+\rho^{-1})$;  set $f_{\pm 1} = t \mp 1$. Note in each case that $f_\rho$ is prime in $\Lambda$.
For $\rho \in S^1$,
consider the ring $\Lambda_{(f_\rho)}$, in which all elements prime to $f_\rho$ have been inverted. It is a local ring with involution.
Also, evaluating $t$ at $\rho$ defines a  homomorphism  between rings with involution, $\Lambda_{(f_\rho)} \rightarrow \cc$, where $\cc$ is the complex numbers equipped with complex conjugation as the involution. If $g \in \Lambda_{(f_\rho)}$, this evaluation will be denoted  $g(\rho)$, as usual.

A matrix $A$ over $\Lambda_{(f_\rho)}$ presents a $\Lambda_{(f_\rho)}$--module. Two matrices related by an invertible row or column operation present isomorphic modules. The torsion submodule of a finitely generated $\Lambda_{(f_\rho)}$--module has an order which is well-defined up to multiplication by units
\cite[\S 1]{milnor}. 
This order can be written as  $(f_\rho)^e$, where  $(f_\rho)^e$ is the maximal power of $f_\rho$ that divides  the product of the  non-zero  entries of a diagonal presentation matrix. We denote this exponent $e_\rho(A)$.

A Hermitian matrix  $H$ over $\Lambda_{(f_\rho)}$ may be evaluated at $t=\omega$  lying in some  open neighborhood of $\rho$ on the unit circle to obtain a Hermitian matrix with  entries in $\cc$.  This is because only a finite number of denominators appear among the  entries of $H$. Thus we can consider the  signature $\sigma( H(\omega))$ for $\omega$
in this neighborhood and the one-sided jumps  at $\rho$ (defined in analogy with the signature jumps of  links in~\S1) which we denote
 by $\ju^\pm (H,\rho)$. Moreover $\sigma( H(\rho))$ and $\ju^\pm (H, \rho)        $ are preserved by
a invertible row operation over $\Lambda_{(f_\rho)}$ followed by corresponding conjugate column operation. We refer to such a pair of operations as a simultaneous row and column operation.

\begin{lem}  \label{diag}  Assume $\rho \ne \pm 1$. Every Hermitian matrix $H$ over $\Lambda_{(f_\rho)}$ can be converted to a diagonal matrix  by  performing a sequence of  simultaneous  row and column  operations.  Moreover  $|\ju^\pm (H, \rho)          |\le  e_\rho(H)$.
\end{lem}

\begin{proof} 
For simplicity of notation in this proof,  let $f$ denote $f_\rho$.

If all  entries of $W$ are divisible by $f$, then factor $f$ out of the matrix and proceed with the simpler matrix $W'$ constructed by dividing all entries of $W$ by $f$.  Since $f$ is symmetric, $W'$ is Hermitian. If $W'$ can be diagonalized as above, the same operations will diagonalize $W$.

If some diagonal entry is not divisible by $f$, it can be used to clear out a column and row to reduce the problem to one for a smaller matrix.

If all the  diagonal  entries are  divisible by $f$ but some  non-diagonal element is not divisible by $f$, arrange, by permuting the rows and columns, that it is the $(2,1)$--entry.  It can be used to clear the first column (not including the $(1,1)$--entry), and the conjugate operations clear the top row.  
The top left $(2,2)$--block looks like
$$ 
\left(
\begin{array}{cc }
 a f^k &    \overline{b}    \\
b &   c   f^m
\end{array}
\right).
$$
Notice that $b(\rho) \ne 0$ since $b$ is not divisible by $f$. 
 We will choose an $\alpha$, add $\alpha$ times the second row to the first row and add $\overline{\alpha}$ times the second column to  the first column,  making  the $(1,1)$--entry 
 $$  a f^k +\alpha b + \overline{\alpha b} + \alpha \overline{\alpha} c f^m.$$
 We want to choose $\alpha$ so that $\alpha b + \overline{\alpha b}$ evaluated at $\rho$ is nonzero.  If $b(\rho) + \overline{b(\rho)} \ne 0$, let $\alpha = 1$.  If $b(\rho) + \overline{b(\rho)} = 0$, then let $\alpha = t$.  We claim $\rho b(\rho) + \overline{\rho}\overline{b(\rho)} \ne 0$.  Suppose otherwise.  Then substituting $\overline{b(\rho)} = -b(\rho) $ yields $\rho b(\rho) - \overline{\rho} b(\rho) = 0$, so $\rho - \overline{\rho} = 0$.  But $\rho$ is not fixed by the involution, and we have reached a contradiction.
In this way, we can obtain a diagonal entry not divisible by $f.$

Continuing in this way we obtain a diagonal matrix $\dd$. As $\dd$ is obtained from $H$ by simultaneous row and column operations, $e_\rho(H)= e_\rho( \dd)$ and
$\ju^\pm(H,\rho)= \ju^\pm(\dd,\rho)$.
Finally it is easy to see that the contribution of a diagonal entry of $\dd$ to $\ju^\pm(\dd,\rho)$ is $\pm 1$ if the multiplicity of $\rho$ as a root of this diagonal entry is non-zero. Otherwise the contribution of a diagonal entry is zero.
As $e_\rho(\dd)$ is the sum over the diagonal positions of the multiplicity of $\rho$ as a root of these entries, $|\ju^\pm( \dd, \rho)    |\le  e_\rho(\dd)$. The result follows. 
\end{proof}


\section{Proof of Theorem~\ref{main}} \label{sec.main}
   
Let $V$ be an $(n \times n)$ Seifert matrix for a link.   The matrix $tV-V^\trans$ presents a $\Lambda$--module $M$.
We can perform independent  row and column operations to diagonalize $tV-V^{\trans}$, ending with a diagonal matrix with diagonal entries $[d_1, \ldots, d_k, 0, \ldots, 0]$, where $d_i \ne 0$ and $d_i $ divides $d_{i+1} $ for $1\le i \le k-1$.  This gives a    decomposition:
$$ M   = {\Lambda}^{n-k}\  \ooplus_{i=1}^k  \frac{ \Lambda}{\left< d_i\right>}  .$$
It follows that  
 $A_L= \Delta_{n-k+1}(L) \sim_{\Lambda} \prod_i d_i$;  
 where we let $\sim_R$ mean equal up to a multiplicative factor from the units of a ring $R$.

 The matrix $W = (1-t)V + (1-t)V^\trans$ presents a 
$\Lambda$--module, say,  $N$.  Since $$W = t^{-1}(1-t)(tV - V^\trans),$$ the same row and column operations that diagonalize $tV - V^\trans$ also diagonalize $W$ and we see that    
 $$ N= \Lambda^{n-1}\   \ooplus_{i=1}^k  \frac{ \Lambda}{\left< (1-t)d_i\right>}.$$


 \subsection{Step function} \label{subsec.step}  
 
We first want to observe that, as in the case of knots, the signature function is a step function.  
There is a  diagonalization of $W$ over the field of fractions $\rr(t)$.  That is, there is a determinant one matrix $A(t)$ with entries in $\rr(t)$ such that $A(t) W(t) {A^\trans}(t^{-1})$ is diagonal,   with diagonal entries rational functions: $[q_1, \ldots, q_k, 0, \ldots 0]$.  (Since the module $N$ described above becomes an $n-k$ dimensional vector space with $\rr(t)$ coefficients, we can use the same value of $k$ here as above.)  For all but the finite set of $\rho \in S^1$ for which $A(\rho)$ is not defined, this provides a diagonalization of $W(\rho)$.  Away from this set of singular values, we see that the signature of $W$ can jump only at zeroes and poles of the diagonal entries. Thus, the signature function is a step function; in particular, it has a finite number of discontinuities.  


\subsection{Jumps away from $\rho = \pm 1$\label{awayfpm1}}

Let $\rho \in S^1$ be a fixed complex number.  By Lemma~\ref{diag}, we can diagonalize $W$ over the ring $\Lambda_{(f_\rho)}$.  Since this matrix presents the  module $N \otimes \Lambda_{(f_\rho)}$, we have that after reordering the entries, this  diagonalization  of $W$ has entries $[\alpha_1 (f_\rho)^{\epsilon_1}, \ldots , \alpha_k (f_\rho)^{\epsilon_k}, 0, \ldots , 0]$ where the $\alpha_i$ are units in $\Lambda_{(f_\rho)}$ and 
 $(f_\rho)^{\epsilon_i}$ is the maximum power of $f_\rho$ dividing $d_i$.

The jumps at the discontinuities, $\ju^\pm$, arise from the diagonal terms for which $\epsilon_i >0$.  It is now evident that these jumps are bounded by the multiplicity of $\rho$ in $A_L$, as was to be proved.
 Thus $\sigma_L(\omega)$ is a step function as claimed.


 \subsection{Jump   at $-1$} \label{subsec.jump-1} 
 
 The diagonalization lemma, Lemma~\ref{diag}, does not apply for $\rho = -1$ 
 (see ~\S \ref{counter}).  
 A transformation corrects for this.  Let $W^*$ be the Hermitian matrix $(1-t^2)V + (1- t^{-2})V^\trans$.  The jumps of the signature function of $W$ at $-1$ correspond to the jumps of the signature function of $W^*$ at $\rho = \sqrt{-1}$.  Notice that $f_\rho = t^{-1} + t$.
 
 The diagonalization of $W^*$ in $\Lambda_{({f_\rho})}$ has nonzero entries of the following types:  
 \begin{itemize}
 \item $\alpha_i$, where $\alpha_i$ is a unit.\vskip.05in
 \item $\beta_i (f_\rho)^{b_i}$, where $\beta_i(\rho) >0$ and $b_i$ is odd.\vskip.05in
  \item $\gamma_i (f_\rho)^{c_i}$, where $\gamma_i(\rho) < 0$ and $c_i$ is odd.\vskip.05in
  \item $\delta_i (f_\rho)^{d_i}$, where $\delta_i(\rho) >0$ and $d_i$ is even.\vskip.05in
  \item $\eta_i (f_\rho)^{e_i}$, where $\eta_i(\rho) <0$ and $e_i$ is even.\vskip.05in
  \end{itemize}
 
 Suppose that the number of elements for type $\beta$, $\gamma$, $\delta$ and $\eta$ are given by $B$, $C$, $D$, and $E$, respectively.  The total jump in the signature function of $W$ at $-1$ is 0, by conjugation symmetry, so the same is true for the total jump in the signature function of $W^*$ at $\sqrt{-1}$.  This implies that $B = C$.  
 We then see that $|\ju^\pm(W^*,\sqrt{-1})| \le |D - E |$. Since $D$ and $E$ both correspond to diagonal elements for which $f_{\rho}$ has even exponent, it is now clear that the jump is at most one half the multiplicity of $-1$ in $A_L$. It   follows that  $\mult_{-1}(A_L)$ is even and  $|\ju^\pm(W, -1)| \le (1/2)\mult_{-1} (A_L)$, as was to be proved. 
 
 
 \subsection{Jump at $1$} \label{subsec.jump1}

Since $\sigma_L(1)=0$, $|\ju^\pm(1)| =| \sigma_L^\pm(1)|$.
We wish to show that $ \mu_L-1$ is an upper bound for $| \sigma_L^\pm(1)|$. We let $t = \cos(2\theta) + \i\sin(2\theta)$.  Expressing $W(t) = (1-t)V + (1-t^{-1})V^\trans$ in terms of $\theta$ and simplifying, we find
$$W(t) = -\sin(2 \theta)  \left( \i(V - V^\trans) - \tan(\theta)(V+V^\trans) \right).$$

The matrix  $V-V^\trans$ gives the intersection form for the Seifert surface. Thus we may take $V-V^\trans$ to be the direct sum of  $g$ copies of  $[\begin{smallmatrix} 0&1\\ -1&0 \end{smallmatrix}]$ direct sum with a $(\mu_L-1) \times (\mu_L-1)$ zero matrix.
It follows that the  form  $i(V - V^\trans)$ has 
$g$ eigenvalues $1$ and $g$ eigenvalues $-1$. 
After a small perturbation, the number of positive and negative eigenvalues will both continue to be at least $g$, so that  the absolute value of the signature is at most $\mu_L -1$.  (Alternatively, the matrix $\i(V - V^\trans)$ is congruent to a diagonal matrix with its first $2g$ diagonal entries alternating between $1$ and $-1$.   
Thus, there are exactly $g$ sign changes in the  sequence of the first  leading $2g$ principle minors of this congruent matrix. 
By the continuity of determinants, the same is true for a small perturbation.   
It follows that the signature of the upper left $2g \times 2g$ block of a small perturbation is also zero. This upper left $2g \times 2g$ block is nonsingular.
It follows that  the absolute value of the signature of a small perturbation is at most $\mu_L-1$.)


\section{Proof of Theorem~\ref{bound2}} \label{sec.bound2}

\begin{lem}\label{estmult} If $\Delta_L(t) \ne 0$, then  $\mult_1\left(\Delta_L(t)\right) \ge  \mu_L-1$.
\end{lem}

\begin{proof}
Since $ \Delta_L(t) \ne 0$, the $\Lambda$--module $M$ presented by $V - tV^\trans$, has no free summand: $M \cong \oplus_{i=1}^n  \frac{\Lambda}{\left< d_i \right>}$ where each $d_i $ divides $d_{i+1}$ and is nonzero. Tensoring with $\rr$, viewed as a $\Lambda$--module with trivial action (formally, $\rr \cong \frac{\Lambda}{\left< 
1-t \right>
}
$), yields the $\rr$--vector space presented by $V - V^\trans$.  For a link of $\mu$ components, this is isomorphic to $\rr^{\mu -1}$ 
(see the description of $V-V^\trans$ in \S\ref{subsec.jump1}). 
 However, $\frac{\Lambda}{\left< d_i \right>} \otimes   \frac{\Lambda}{\left< 
 1-t
\right>}$ is trivial unless $d_i$ is divisible by $(1-t)$.  Thus,  $(\mu -1)$ of the $d_i$ are divisible by $(1-t)$.  This implies the statement of the lemma.

\end{proof}

By Theorem~\ref{main}, $\mu_L-1 \ge |\ju^\pm(1)|$.   Using the triangle inequality, we easily obtain Theorem~\ref{bound2} from Corollary~\ref{bound}.

\begin{rem} { \em For $L= 8n8(0,0,0)$ or $8n8(1,0,0)$, discussed in~\S \ref{L8n8} below, we have  $\Delta_L(t) = 0$, and 
$\mult_1(A_L(t))= 2 <3= \mu_L-1.$ Thus Lemma~\ref{estmult} cannot be modified 
to read  $\mult_1(A_L(t)) \ge  \mu_L-1$ in the situation where $\Delta_L(t) = 0.$
By  adapting the argument given above, one can show  that $$\mult_1(A_L) \ge \mu_L-h_L.$$
Note that  Proposition~\ref{even} says that this inequality is also a congruence modulo two.
}
\end{rem}

   
\section{Proof of Theorem~\ref{cong}} \label{sec.cong}

Let $\rho \ne \pm 1$.
If $\d_i$  is an entry of the diagonal matrix 
$\dd$ appearing in 
\S \ref{awayfpm1}
and has $\mult_\rho(\d_i)$  odd, then $\d_i$ contributes $\pm 2$ to $\ju(\rho)$. 
If $\d_i$   is an entry of 
$\dd$ and has $\mult_\rho(\d_i)$  even, then $\d_i$ contributes $0$ to $\ju(\rho)$.
If $\mult_\rho(A_L)$ is odd, then an odd number of entries $\{\d_i\}$ have $\mult_\rho(\d_i)$  odd.
If $\mult_\rho(A_L)$ is even, then an even number of entries $\{\d_i\}$ have $\mult_\rho(\d_i)$  odd.
The result stated in the first sentence follows.

If $\rho \ne \pm 1$, and $\mult_\rho(A_L)=1$, then $\mult_\rho(\d_i)$  is nonzero for exactly one $i$.
The result stated in the second sentence follows.


\section{Proof of Proposition \ref{even} }\label{evenp}

Let $W^{**}=(1-t^4)V + (1- t^{-4})V^\trans$. The jumps of the signature function of $W$ at $1$ correspond to the jumps of the signature function of $W^{**}$ at $\rho=\sqrt{-1}$.
Thus it
 follows that $$\ju^+(W^{**},\sqrt{-1})= -\ju^-(W^{**},\sqrt{-1}).$$   Notice that $(1-t^4)= (1-t)(1+t)(1+t^2)$ and $f_\rho = t^{-1} + t$.
 Moreover, up to $\sim_\Lambda$, $f(t)=1-t$ is the only irreducible polynomial  for which $f(t^4)$ has $1+t^2$ as a factor.  
It follows that $e_{t-1}(W)=e_{t+t^{-1}}(W^{**})$.
The  argument  in \S \ref{subsec.jump-1} shows that $e_{t+t^{-1}}(W^{**})$ is even.
Thus $e_{t-1}(W)$ is even. But  $$e_{t-1}(W)=e_{t-1}((1-t)(tV-V^\trans))= e_{t-1}(tV-V^\trans)+ 2g +\mu_L-h_L.$$
As $\mult_1(A_L)= e_{t-1}(tV-V^\trans)$, the result follows.


\section{A module approach to the signature function}\label{module}

 The proof  of Lemma~\ref{diag} and of Theorem~\ref{main}  in the case of $\rho = -1$ (as given in \S \ref{subsec.jump-1})   point  to a general approach to understanding the signature function.  Let $V$ be an $n\times n$ Seifert matrix for a link $L$.  The $\Lambda$--module $M$ presented by $V - tV^\trans $ has a direct sum composition     
$$M\cong \frac{\Lambda}{\left< d_1\right>} \oplus  \frac{\Lambda}{\left< d_2\right>} \cdots \oplus   \frac{\Lambda}{\left< d_k\right>} \oplus \Lambda^{n-k},$$  
where the $d_i \in \Lambda$ are nonzero and $d_i $ divides $d_{i+1}$ for all $i < k$.  We have that $A_L  \sim_\Lambda \prod d_i$; also, 
$\Delta_L \ne 0$ if and only if $n = k$.  
  More generally, $h_L=n-k+1$. 
Let $X$ denote the infinite cyclic cover of $S^3 \setminus L$ specified by the linking number.
According to \cite[Thm 6.5]{lickorish}, the $\zz[t,t^{-1}]$--module $H_1(X,\zz)$  is presented by the matrix  $V - tV^\trans$. 
It follows that $M$ is a description of $H_1(X,\rr)$ as a $\Lambda$--module.

For any $\rho \in S^1$,  let $\phi_o(L, \rho)$ and $\phi_e(L,\rho)$ to be the number of $d_i$ in which $f_\rho$ has a positive
 exponent that is odd or even, respectively. Bounds on the signature function are easily 
given
in terms of these functions; the proofs follow along the same lines as our earlier work.  
Here is an example of the type of result that can be attained in this way.

\begin{thm} If $\rho \ne \pm 1$, then $|\ju(\rho)| \le 2 \phi_o(L, \rho)$ and $|\ju(\rho)| = 2 \phi_o(L, \rho)$ modulo 4. Also $|\ju_\pm(-1)| \le   \phi_e(L, -1)$ and $|\ju_\pm(-1)| =  \phi_e(L, -1)$ modulo 2. 
Thus $$ | \sigma_L - \sigma_L^+(1) | \le     \phi_e(L,-1) + \sum_{\rho \in S^1 \setminus\{-1,1\}} \phi_o(L,\rho).$$
\end{thm}
                    

\section{Some examples} \label{ex} 

 In this section we will present 
some examples that  illustrate our results.
Notice that all the polynomials presented are in $\zz[t,t^{-1}]$ and are defined up to multiplication by a unit, $\pm t^i$.  When possible, we normalize   polynomials $f$ so that   $f \in \zz[t]$ and $f(0)\ne  0$.


\subsection{A family of links with zero Alexander polynomial}  

Consider the family of 3--component links $L_n$ illustrated in  Figure~\ref{figure}.

\begin{figure}[h]
\centerline{\includegraphics[width=2in]{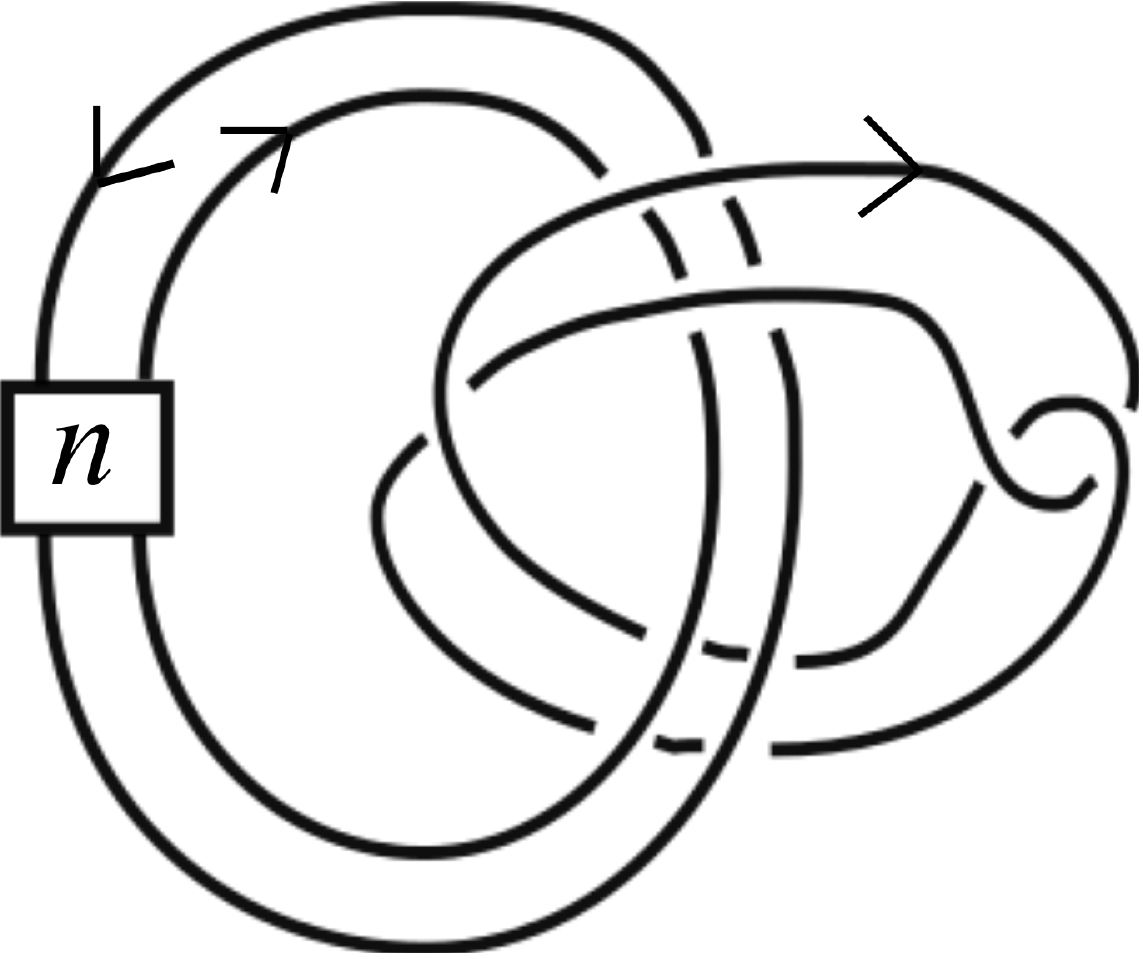}}
\caption{Link with $n$ denoting full twists.}
\label{figure}
\end{figure}

There is an obvious disconnected Seifert surface consisting of an annulus  and a punctured torus.
We tube these together to obtain a connected Seifert surface with Seifert matrix
$$V_n=\left( \begin{array}{cccc}
-1 & 1 & 0 & 0 \\
0 & -1 & 1 &0\\
0 & 1 & n &0\\
0 & 0& 0& 0 \end{array} \right).$$

We have that $\Delta_{L_n}=0$;
$A_{L_n}=\Delta_2(L_n)= 
(t-1)((n+1) t^2-(n+2) t+(n+1))$; 
$$\sigma_{L_n}= 
\begin{cases} -3 &\mbox{if } n \le -2 \\
-1 &\mbox{if } n \ge -1\
. \end{cases}$$

Thus $A_{L_0}=(t-1)^3$ and $\sigma_{L_0}=-1$.   As $A_{L_0}$ has no roots on $S^1\setminus \{1\}$, $\sigma_{L_0}$ is identically $-1$ on $S^1\setminus \{1\}$.

Similarly,  $A_{L_{-1}}=(t-1)$ and $\sigma(L_{-1})=-1$.   As $A_{L_{-1}}$ has no roots on $S^1\setminus \{1\}$, $\sigma_{L_{-1}}$ is identically $-1$ on $S^1\setminus \{1\}$.

For $n \notin \{0,-1\}$,  $A_{L_n}$ has a single conjugate pair of roots $\alpha_n$ and $ \bar \alpha_n$  on $S^1\setminus \{1\}$ with real part $\Re (\alpha_n)=\frac {n+2}{2(n+1)}$, where
$\sigma_{L_n}$ must make a total jump of $\pm2$. The one-sided jump at $1$ must have absolute value less than or equal to $2$. It follows that 
if $n \le -2$, then
$$\sigma_{L_n}(e^{\pi i x})= \begin{cases} 0 &\mbox{if } x=0 \\
-1 &\mbox{if } 0 <x < \frac 1 \pi \arccos( \frac {n+2} {2(n+1)} )\\
-2 &\mbox{if } x=\frac 1 \pi \arccos( \frac {n+2} {2(n+1)} )\\
-3 &\mbox{if } \frac 1 \pi \arccos( \frac {n+2} {2(n+1)} ) <x\le 1 .
 \end{cases}$$

If $n \ge 1$, then
$$\sigma_{L_n}(e^{\pi i x})= \begin{cases} 0 &\mbox{if } x=0 \\
1 &\mbox{if } 0 <x < \frac 1 \pi \arccos( \frac {n+2} {2(n+1)} )\\
0 &\mbox{if } x=\frac 1 \pi \arccos( \frac {n+2} {2(n+1)} )\\
-1 &\mbox{if } \frac 1 \pi \arccos( \frac {n+2} {2(n+1)} ) <x\le 1 .
 \end{cases}$$


\subsection{Low crossing links with $\Delta_L(-1) =0$}

In   Thistlethwaite's list of links with ten or fewer crossings, there are   16 links for which  $\Delta_L(-1)= 0$; these are
$L8n6$, $L8n8$, $L9n18$, $L9n19$, $L9n27$, $L10n32$, $L10n36$, $ L10n56$, $ L10n57$, $L10n59$, $L10n91$, $L10n93$, $L10n94$, $L10n104$, $L10n107$,  and $L10n111$.  
We will investigate the first three of these  links using the Seifert matrices given by LinkInfo~\cite{cha-livingston}. 
In LinkInfo,   links are equipped with specified orientations and all orientations, up to simultaneous reversal of all components, are considered. These orientations are described by a zero-one vector  of length $\mu-1$.


\subsubsection{$L8n6$}

We study the 3--component  link $L8n6$ with its different orientations.  

Consider first $L8n6 (0, 0)$.  One computes $\sigma_L(-1)=3$,  $A_L=\Delta_L =(t+1)^2(t-1)^2$. 
 Thus the signature function can only jump at $\pm1$.  It follows that the one-sided jump at 1 can be at most 2. Also, $|\ju^\pm(-1)| \le 1$.   The only possible signature function consistent with the above is given by
$$ \sigma_L(e^{\pi i x})= \begin{cases} 0 &\mbox{if } x=0 \\
2 &\mbox{if } 0 <x <  1\\
3 &\mbox{if } x=1  
. \end{cases}$$

Next, we consider $L8n6 (1, 0)$.  
 One computes $\sigma_L(-1)= -5$, $A_L= \Delta_L=(t-1)^2(t+1)^2 \left(t^2-t+1\right)$. The last factor is the cyclotomic polynomial with roots the primitive 6th roots of unity. Thus the signature function can only jump at these  6th roots of unity and at $\pm1$. The total jump at the  6th roots of unity must be $\pm2$. The one-sided jump at 1 can be at most 2. Also, $|\ju^\pm(-1)| \le 1$. The only possible signature function consistent with the above is given by
$$ \sigma_L(e^{\pi i x})= \begin{cases} 0 &\mbox{if } x=0 \\
-2 &\mbox{if } 0 <x <  \frac 1 3\\
-3 &\mbox{if } x=\frac 1 3\\
-4 &\mbox{if } \frac 1 3 <x<1  \\
-5 &\mbox{if } x=1  
. \end{cases}$$

Finally, consider $L8n6 (0, 1)$.  
 One computes $\sigma_L(-1)= -1$,  $A_L=\Delta_L= (t-1)^2 (t+1)^2$. Thus the signature function can   jump only at $\pm1$.  The one-sided jump at 1 can be at most 2. Also, $|\ju^\pm(-1)| \le 1$.  After calculating $\sigma_L(\i)= 0$, we conclude that
  $$ \sigma_L(e^{\pi i x})= \begin{cases} 
0 &\mbox{if } 0 \le x <  1\\
-1 &\mbox{if } x=1  
. \end{cases}$$

The   computations and reasoning for $L8n6 (0, 1)$   also hold  for $L8n6 (1, 1)$.


\subsubsection{$L8n8$}\label{L8n8}  
The link $L8n8(0,0,0)$ has  four components.  It is built by  replacing each component of a Hopf link with a pair of unlinked oppositely oriented components. It has a  disconnected Seifert surface composed of two annulli. One can easily work out, by hand, that $\Delta_L=0$, $A_L = \Delta_2(L) =(t-1)^2$, and $\sigma_L(\omega)=0$, for all $\omega$.  

The link $L8n8(1,0,0)$  is constructed from a   Hopf link by  replacing one component with  a pair of unlinked  components oriented in the same way and replacing the other component with a pair of unlinked oppositely oriented components.  The computation in this case are most easily carried out with computer assitance.  They yield the following:
$\Delta_L=0$, $A_L = \Delta_2(L) =(t-1)^2$, and $\sigma_L(-1)=0$. One then can conclude that $\sigma_L(\omega)=0$, for all $\omega$.

For $L8n8(1,0,1)$, which is obtained by taking a positive Hopf link and replacing each component by a pair of unlinked  components oriented in the same way, 
 $A_L=\Delta_L=(t-1)^3(t+1)^2$, $\sigma_L(-1)=-4$.  The only possible signature function consistent with the above is given by
$$ \sigma_L(e^{\pi i x})= \begin{cases} 0 &\mbox{if } x=0 \\
-3&\mbox{if } 0 <x < 1\\
-4 &\mbox{if } x=1  
. \end{cases}$$
The other orientations on $L8n8$ are obtained from the above three by symmetries (some reversing the ambient orientation), and so we do not consider these  other orientations.


\subsubsection{$L9n18$}

 Consider the 2--component link $L9n18 (0)$. 
 One computes  from the Seifert matrix that $A_L=  \Delta_L=(t-1) (t+1)^2 \left(t^2-t+1\right)^2$ and $\sigma_L= 6$. The last factor is the square of the cyclotomic polynomial with roots the primitive 6th roots of unity. Thus the signature function can only jump at these  6th roots of unity and at $\pm1$. The total jump at the  6th roots of unity are either  zero, or $\pm4$. The one-sided jump at 1 can be at most 1. Also, $|\ju^\pm(-1)| \le 1$. The only possible signature function consistent with the above is given by
$$ \sigma_L(e^{\pi i x})= \begin{cases} 0 &\mbox{if } x=0 \\
1 &\mbox{if } 0 <x <  \frac 1 3\\
3 &\mbox{if } x=\frac 1 3\\
5 &\mbox{if } \frac 1 3 <x<1  \\
6 &\mbox{if } x=1  
. \end{cases}$$

 Consider $L9n18 (1)$.
One computes $\sigma_L= -2$ and $A_L(t) =\Delta_L= (t-1) (t+1)^2$. Thus the signature function can only jump  at $\pm1$. The one-sided jump at 1 can be at most 1. Also, $|\ju^\pm(-1)| \le 1$. The signature function consistent with the above is given by
$$ \sigma_L(e^{\pi i x})= \begin{cases} 0 &\mbox{if } x=0 \\
-1 &\mbox{if } 0 <x <  1\\
-2 &\mbox{if } x=1  
. \end{cases}$$


\subsection{A  3--component link with $\Delta_L =(t-1)^2 (t+1)^2$ and zero signature function}\label{counter}

Consider the  Seifert matrix 
\[ V =\left( \begin{array}{cccc}
1 & -1 & 1 & 1 \\
0 & 0 &0 &-1\\
1 & 0 & 0 &2\\
1 & -1& 2 & 0 \end{array} \right).\]
Pick a 3--component link $L$ with this Seifert matrix. One computes $\Delta_L =(t-1)^2 (t+1)^2$ and $\sigma_L(\i)=\sigma_L(-1)=0.$ Since $\sigma_L(\omega)$ can only jump at $\pm 1$, we see that $\sigma_L(\omega)$ is identically zero.  

This example demonstrates  that  Lemma~\ref{diag} cannot be extended to the case $\rho=-1$. Suppose $W$ can diagonalized as a Hermitian form over $\Lambda_{(t+1)}$. Then $W$ diagonalized must have all entries self-conjugate. Thus, any entry divisible by $1+t$ must be divisible by an even power  of $t+1$.  It follows that  exactly one entry can be divisible by $1+t$, and this entry must have exponent two. However, this would imply that  $\sigma_L(\omega)$ has  a one-sided jump at $-1$, yielding the desired contradiction.


\end{document}